\documentclass[12pt]{amsart}

\usepackage[utf8]{inputenc}
\usepackage{amsfonts}
\usepackage{latexsym}
\usepackage{amssymb}
\usepackage{amsmath}

\usepackage{hyperref}
\hypersetup{
    colorlinks,
    linkcolor=CadetBlue,
    citecolor=CadetBlue,
    urlcolor=CadetBlue
}

\usepackage[left=1.9cm,top=2.5cm,bottom=2.5cm,right=1.9cm]{geometry}

\usepackage[dvipsnames]{xcolor}
\usepackage{enumerate}
\usepackage{mathrsfs}
\usepackage{todonotes}

\newtheorem{thm}{Theorem}[section]
\newtheorem{cor}[thm]{Corollary}
\newtheorem{lemma}[thm]{Lemma}

\newtheorem{assu}{Assumption}

{
\theoremstyle{definition}
\newtheorem{example}[thm]{Example}

}

\begin{document}

\title[]{A fast Berry-Esseen theorem under minimal density assumptions}



\author[S. Johnston]{Samuel G. G. Johnston}
\address{Department of Mathematics, Strand Building, King's College London, London, WC2R 2LS, United Kingdom.} \email{samuel.g.johnston@kcl.ac.uk}


\maketitle

\begin{abstract}
Let $X_1,\ldots,X_N$ be i.i.d.\ random variables distributed like $X$. Suppose that the first $k \geq 3$ moments $\{ \mathbb{E}[X^j] : j = 1,\ldots,k\}$ of $X$ agree with that of the standard Gaussian distribution, that $\mathbb{E}[|X|^{k+1}] < \infty$, and that there is a subinterval of $\mathbb{R}$ of width $w$ over which the law of $X$ has a density of at least $h$. Then we show that 
\begin{align} \label{eq:bnew}
\sup_{s \in \mathbb{R}} \left| \mathbb{P} \left( \frac{X_1 + \ldots + X_N}{ \sqrt{N} } \leq s \right) -  \int_{-\infty}^s \frac{ e^{ - u^2/2} \mathrm{d} u }{ \sqrt{2 \pi }} \right| \leq 3 \left\{ \frac{\mathbb{E}[|X|^{k+1}]}{  N^{  \frac{k-1}{2}} } +  e^{ - c hw^3 N/\mathbb{E}[|X|^{k+1}] } \right\},
\end{align}
where $c > 0$ is universal. By setting $k=3$, we see that in particular all symmetric random variables with densities and finite fourth moment satisfy a Berry-Esseen inequality with a bound of the order $1/N$. 

Thereafter, we study the Berry-Esseen theorem as it pertains to perturbations of the Bernoulli law with a small density component, showing by means of a reverse inequality that the power $hw^3$ in the exponential term in \eqref{eq:bnew} is asymptotically sharp. 
\end{abstract}
\color{black}
\section{Introduction}
\subsection{The Berry Esseen theorem}
We say that a real-valued random variable $X$ is centered if $\mathbb{E}[X]=0$ and $\mathbb{E}[X^2]=1$. If $X$ is any random variable with mean $\mu$ and variance $\sigma$, $\frac{X-\mu}{\sigma}$ is centered.

Let $X_1,\ldots,X_N$ be independent random variables distributed like $X$, where $X$ is centered and $\mathbb{E}[|X|^3] < \infty$.
The central limit theorem states that the sum $N^{-1/2}\sum_{i=1}^N X_i$ converges in distribution to a standard Gaussian random variable.

The celebrated Berry-Esseen inequality \cite{berry, esseen} provides a quantitative version of the central limit theorem, stating that
\begin{align} \label{eq:be}
\sup_{s \in \mathbb{R}} \left| \mathbb{P} \left(  \frac{X_1 + \ldots + X_N}{ \sqrt{N} } \leq s \right) -  \int_{ - \infty}^s e^{ - u^2/2} \frac{ \mathrm{d} u }{ \sqrt{2 \pi}} \right| \leq \frac{ C \mathbb{E}[|X|^3]}{ \sqrt{N}},
\end{align}
where $C$ is a universal constant not depending on the distribution of $X$.

For general centered random variables with finite third moment, the $1/\sqrt{N}$ speed of convergence in \eqref{eq:be} is asymptotically sharp.
To see this, consider the case where $X_1,\ldots,X_N$ are independent Bernoulli random variables with probabilities $\mathbb{P}( X_i = -1) = \mathbb{P}( X_i = 1 ) = 1/2$. Then the distribution function of the renormalised sum $N^{-1/2}(X_1 +\ldots + X_N)$ has jumps of magnitude $1/\sqrt{N}$.

In this article we will discuss how under some fairly mild conditions, one may nonetheless improve on the $1/\sqrt{N}$ rate of convergence in \eqref{eq:be}. We will further be concerned with finding the minimal conditions under which one can improve this rate of convergence.

In the sequel it will lighten notation to write
\begin{align*}
\mathrm{d}_{\mathrm{KS}}(X,Y) := 
\sup_{s \in \mathbb{R}} \left| \mathbb{P} \left(  X \leq s \right) -   \mathbb{P} \left(  Y\leq s \right) \right| 
\end{align*}
for the Kolmogorov-Smirnov distance between random variables $X$ and $Y$. Furthermore, throughout $G$ will refer to a standard Gaussian random variable, so that
\begin{align*}
\mathbb{P}(G \leq s ) =\int_{-\infty}^s \frac{ e^{-u^2/2} \mathrm{d}u }{ \sqrt{2\pi}}.
\end{align*}

\subsection{A fast Berry-Esseen inequality.}
We will see that if further moments (i.e.\ third and possibly higher) of the distribution of $X$ match that of the standard Gaussian distribution, and that the law of $X$ is not singular with respect to Lebesgue measure, then we can improve on the $1/\sqrt{N}$ speed of convergence afforded by the standard Berry-Esseen inequality. 

Recall that every probability measure $\mu$ on the real line has a unique decomposition
\begin{align} \label{eq:decomp}
\mu = \mu_d + \mu'
\end{align}
so that $\mu_d$ is absolutely continuous with respect to Lebesgue measure, and 
$\mu'$ is singular with respect to Lebesgue measure. The Radon-Nikodym derivative of $\mu_d$ against Lebesgue measure is the nonnegative function $f:\mathbb{R} \to [0,\infty)$ satisfying
\begin{align} \label{eq:diffuse}
\mu_d ([a,b)) = \int_a^b f(u) \mathrm{d}u
\end{align}
for every subinterval $[a,b)$ of the real line.

We work under the following assumption. 
\begin{assu} \label{assu:main}
For some $h,w \leq 1$, the density of the part of $\mu$ absolutely continuous with respect to Lebesgue measure contains a rectangle of width $w$ and height $h$. In other words, in the setting of \eqref{eq:decomp} and \eqref{eq:diffuse} for some $a \in \mathbb{R}$ we have
\begin{align*}
f(s) \geq h \qquad \text{for all $s$ in $[a,a+w)$}.
\end{align*}
\end{assu}

We now state our first result.
\begin{thm} \label{thm:main}
Let $X_1,\ldots,X_N$ be independent random variables distributed like $X$, where $X$ has law $\mu$ satisfying Assumption \ref{assu:main}. Suppose that for some $k \geq 3$ we have 
\begin{align} \label{eq:mm}
\mathbb{E}[X^j] =  \int_{-\infty}^\infty u^j \frac{ e^{-u^2/2} \mathrm{d}u }{ \sqrt{2\pi}}, \qquad \text{for } j=1,\ldots,k.
\end{align}
Then provided $\mathbb{E}[|X|^{k+1}]<\infty$ we have
\begin{align} \label{eq:main}
\mathrm{d}_{\mathrm{KS}}\left(\frac{X_1+\ldots+X_N}{\sqrt{N}},  G\right) \leq 3 \left\{ \frac{\mathbb{E}[|X|^{k+1}]}{N^{\frac{k-1}{2}}}  +  e^{ - \frac{1}{160} hw^3 N/\mathbb{E}[|X|^{k+1}]} \right\}.
\end{align}
\end{thm}
The exact statement of Theorem \ref{thm:main} is new, though as we discuss in the sequel, a result of this form would come as no surprise to specialists in Berry-Esseen theory. We will discuss in the sequel how, after appealing to Lemma \ref{lem:charcont} of the present article, it is possible to derive a statement similar to \eqref{eq:main} using Osipov's theorem \cite[Theorem 5.18]{petrov}.

Theorem \ref{thm:main} is not the sharpest possible result we give. In the sequel we will give a more refined (but more complicated) inequality in place of \eqref{eq:main} with a leading constant that decreases as $k$ increases, and an improved constant in the exponential term. See Theorem \ref{thm:main2} for our sharpest statement.

At least in the case $k = 3$ of \eqref{eq:mm}, the matching of the first three moments of a random variable $X$ with the standard Gaussian distribution is a natural condition: it holds for all centered symmetric random variables. (A random variable is symmetric if $-X$ has the same law as $X$.) Highlighting this case in particular, we have the following corollary:

\begin{cor} \label{cor:symmetric}
Let $X_1,\ldots,X_N$ be independent symmetric unit-variance random variables distributed like $X$, where the law of $X$ satisfies Assumption \ref{assu:main}. Suppose that $\mathbb{E}[X^4]<\infty$. Then we have 
\begin{align} \label{eq:main3}
\mathrm{d}_{\mathrm{KS}}\left(\frac{X_1+\ldots+X_N}{ \sqrt{N}},  G\right) &\leq 3 \left\{ \frac{ \mathbb{E}[X^4]  }{N} + e^{ - \frac{1}{160} hw^3 N / \mathbb{E}[X^4]} \right\}.
\end{align}
\end{cor}
\begin{proof}
If $-X$ has the same law as $X$, then $\mathbb{E}[X]=\mathbb{E}[X^3]=0$. Thus since $X$ has unit variance, \eqref{eq:mm} holds with $k=3$. Now use Theorem \ref{thm:main}.
\end{proof}


\subsection{The sharpness of Theorem \ref{thm:main} and a reverse inequality}
With Theorem \ref{thm:main} at hand, the second part of the present article is concerned with establishing minimal conditions involving a density function under which we can improve the Berry-Esseen inequality from a bound of order $1/\sqrt{N}$ to one of a higher power in $1/\sqrt{N}$. To provide some motivation, we will center the discussion around a perturbation of the Bernoulli density, recalling from the 
from the introduction that the Bernoulli random variable offers a canonical example of the asymptotic sharpness of the $1/\sqrt{N}$ speed of convergence in the standard Berry-Esseen inequality. In short, in this section we study the sharpness of Theorem \ref{thm:main} by adding small density component to the Bernoulli law.

Namely, for $hw \leq 1$, suppose we take a mixture composed of $1-hw$ parts of a rescaled Bernoulli random variable and $hw$ parts of a uniform random variable of width $w$ centered at the origin, where the Bernoulli component is scaled so that the overall variance is $1$. In other words, define the symmetric probability law $\mu_{h,w}$ on $\mathbb{R}$ by 
\begin{align} \label{eq:special}
\mu_{h,w}( \mathrm{d}s ) = (1 - hw) \left( \frac{\delta_x(\mathrm{d}s ) + \delta_{-x}(\mathrm{d}s)}{2} \right) + h \mathrm{1}_{[-w/2,w/2]}(s) \mathrm{d}s,
\end{align}
where $x = x_{h,w}$ is chosen so that $\int_{-\infty}^\infty s^2 \mu_{h,w}(\mathrm{d}s) = 1$. That is, $x = (1 - hw^3/12)^{1/2}(1-hw)^{-1/2}$. 

The probability law $\mu_{h,w}(\mathrm{d}s)$ is clearly constructed so that it satisfies Assumption 1 (though in a sense, minimally so). Moreover, a random variable distributed according to $\mu_{h,w}$ satisfies the moment matching condition \eqref{eq:mm} with $k=3$. In particular, by Corollary \ref{cor:symmetric}, for fixed $h,w$, i.i.d.\ sums of random variables distributed like $\mu_{h,w}$ satisfy a Berry-Esseen theorem of rate $O_{h,w}(1/N)$. 

We now study the small-$(h,w)$ asymptotics of the Berry-Esseen theorem associated with the law $\mu_{h,w}$, which may be used to test the sharpness of Theorem \ref{thm:main}. Namely, first we provide an example which shows that $h,w$ can be taken as small as $N^{-\frac{1}{4}+o(1)}$ while still maintaining the $O(1/N)$ rate of convergence in the Berry-Esseen theorem. Then we again use this law to show that in some sense, the dependence of Theorem \ref{thm:main} on the parameters $h,w$ is asymptotically sharp.

In the former direction, we have the following toy example:

\begin{example} \label{ex:toy}
Write 
$\nu_N$ for the probability law $\mu_{h,w}$ associated with \begin{align*}
h = w = \delta_N := 4 ( \log N /N)^{1/4}.
\end{align*}
Then for $N \geq 100000$, if $X_1,\ldots,X_N$ are independent and identically distributed with law $\nu_N$, we have
\begin{align} \label{eq:main4}
\mathrm{d}_{\mathrm{KS}}\left(\frac{X_1+\ldots+X_N}{\sqrt{N}},  G\right) &\leq 8/N.
\end{align}
\end{example}
Example \ref{ex:toy} says that one need only add a tiny square of width and height of the order $N^{-1/4+o(1)}$ to the density of a Bernoulli random variable to improve its Berry-Esseen inequality from rate $1/\sqrt{N}$ to rate $1/N$. 

The equation \eqref{eq:main4} follows fairly quickly from Corollary \ref{cor:symmetric}, the key idea being that $\nu_N$ satisfies Assumption \ref{assu:main} with $h=w=\delta_N$, and that with this choice of $\delta_N$, the exponential term $ e^{ - \frac{1}{160} hw^3 N / \mathbb{E}[X^4]}$ decays at least as fast as $1/N$. We fill out all of the details of Example \ref{ex:toy} in Section \ref{sec:mainproof}.

\vspace{5mm}
Conversely, it turns out that the probability law $\mu_{h,w}$ can be used to establish the asymptotic sharpness of Theorem \ref{thm:main} in its dependency on $h$ and $w$ in Assumption 1.  In this direction, in Section \ref{sec:reverse} we show that if $X_1,\ldots,X_N$ are independent and identically distributed with law $\mu_{h,w}$, then there is a constant $c> 0$ such that whenever $hw^3N \leq c$ we have
\begin{align} \label{eq:counter}
\mathrm{d}_{\mathrm{KS}}\left( \frac{X_1 + \ldots + X_N}{ \sqrt{N} }, G \right) \geq \frac{c}{\sqrt{N}}.
\end{align}
Compare this with Corollary \ref{cor:symmetric}, which states that for fixed $h,w$, i.i.d.\ random variables with law $\mu_{h,w}$ satisfy a Berry-Esseen bound of rate $O_{h,w}(1/N)$. Roughly speaking, \eqref{eq:counter} then says that in order for the Berry-Esseen bound to improve in rate from $1/\sqrt{N}$ to $1/N$, we need $N$ sufficiently large so that $hw^3N$ exceeds $O(1)$. 

This example gives us some indication that the $hw^3$ dependency in the exponential term in Theorem \ref{thm:main} is nearly optimal, and we ultimately use the inequality in \eqref{eq:counter} to quantify this sharpness through the following result.

\begin{thm} \label{thm:reverse}
Let $\rho,\rho'$ be non-negative reals, at least one of which is strictly positive. Let $C,c > 0$ be constants. Then there exists a natural number $N$, and a probability law $\mu$ satisfying Assumption \ref{assu:main} and with finite $(k+1)^{\text{th}}$ absolute moment $\mathbb{E}[|X|^{k+1}]$ such that
\begin{align} \label{eq:reverse}
\mathrm{d}_{\mathrm{KS}}\left(\frac{X_1+\ldots+X_N}{\sqrt{N}},  G\right) > C \left\{ \frac{\mathbb{E}[|X|^{k+1}]}{N^{\frac{k-1}{2}}}  +  e^{ - c h^{1-\rho}w^{3-\rho'} N/\mathbb{E}[|X|^{k+1}]} \right\}
\end{align}
whenever $X_1,\ldots,X_N$ are i.i.d.\ with law $\mu$.
\end{thm}
Thus there is no improving on the power $hw^3$ in \eqref{eq:main}.

We now explain briefly the source of the $hw^3$ term. The headline is that this follows from multiplying $hw$, which is the total amount of mass of the rectangle, with $w^2$, which is the order of the variance of a random variable uniform on an interval of width $w$. Indeed, take a sum of $N$ copies of a random variable distributed like $\mu_{h,w}$. Then, regarding $\mu_{h,w}$ as a mixture of Bernoullis and uniforms, we would expect of the order $Nhw$ of these random variables to be uniforms. Now an (unscaled) sum of $N$ Bernoulli variables has gaps of the order $1$ in its support. In order for the uniforms of variance $O(w^2)$ to cover these gaps, we need of the order $O(1/w^2)$ of them. Thus we need $N$ large enough so that $Nhw \geq O(1/w^2)$ before the gaps in the support are covered. In other words, we need $N$ sufficiently large so that $Nhw^3 \geq O(1)$ before the contribution from the uniform part of the density starts to significantly smoothen the jumps in the density of an i.i.d.\ sum of Bernoullis. 

\color{black}

\subsection{Further discussion}
Let us take a moment to discuss the main ideas of the proof of Theorem \ref{thm:main}. While in the main body of the paper we keep track of constants explicitly, to expedite the discussion here, throughout this section we will write $C,c>0$ for universal constants that may change from line to line. 

The Fourier-analytic approach to proving the classical Berry-Esseen inequality pivots on the Berry smoothing inequality, which states that if $G$ is a standard Gaussian random variable and $Y$ is a random variable with characteristic function $\psi(t) := \mathbb{E}[e^{itY}]$, then for any $L > 0$ we have 
\begin{align} \label{eq:berry smoothing 1}
\mathrm{d}_{\mathrm{KS}}(Y, G) \leq  \frac{4}{L} +  \frac{1}{\pi} \int_{ - L}^L \frac{| \psi(t) - e^{ -t^2/2} |}{|t|} \mathrm{d} t.
\end{align}
The inequality \eqref{eq:berry smoothing 1} is well known, though it appears to go by various names: in \cite{FMN} it is called Feller's lemma, and in \cite{petrov} it is Esseen's theorem. In any case, one may consult \cite[Section XVI.3]{feller} for a proof of \eqref{eq:berry smoothing 1}.  In fact, we have specialised the inequality slightly for the task at hand. More broadly, the Berry smoothing inequality can be used to control using characteristic functions the Kolmogorov-Smirnov distance between any two random variables $Y$ and $Z$, only one of which need have a density. The constants depend on the supremum of this density. \color{black}

The inequality \eqref{eq:berry smoothing 1} says that in order to control the Kolmogorov-Smirnov distance between a random variable $Y$ and a standard Gaussian random variable $G$, it is sufficient to control the difference between their characteristic functions.

With a view to applying this inequality to sums of i.i.d.\ random variables, we note that 
\begin{align*}
\varphi(t) := \mathbb{E}[e^{itX}] \implies \varphi(t/\sqrt{N})^N = \mathbb{E}[e^{it N^{-1/2}(X_1+\ldots+X_N)}],
\end{align*}
so that by \eqref{eq:berry smoothing 1} we have 
\begin{align} \label{eq:berry smoothing 2}
\mathrm{d}_{\mathrm{KS}}\left( \frac{X_1+\ldots+X_N}{\sqrt{N}}, G \right)  \leq  \frac{4}{L} +  \frac{1}{\pi} \int_{ - L}^L \frac{| \varphi(t/\sqrt{N})^N - e^{ -t^2/2} |}{|t|} \mathrm{d} t.
\end{align}
Thus to obtain a good upper bound on $\mathrm{d}_{\mathrm{KS}}\left(N^{-1/2}(X_1+\ldots+X_N),  G \right)$, one requires control of the difference $| \varphi(t/\sqrt{N})^N - e^{ -t^2/2} |$ over $t$ in a large region $[-L,L]$. 

In Fourier-analytic proofs of the classical Berry-Esseen inequality, a manipulation involving characteristic functions, exploiting the centeredness of $X$ and the third moment condition $\mathbb{E}[|X|^3] < \infty$, tells us that for universal constants $c,C>0$ we have 
\begin{align} \label{eq:ineq}
| \varphi(t/\sqrt{N})^N - e^{ -t^2/2} | \leq C\mathbb{E}[|X|^3]  \frac{ |t|^3}{ \sqrt{N}} e^{ - t^2/4} ~~ \qquad \text{for all $|t| \leq c \frac{ \sqrt{N}}{\mathbb{E}[|X|^3] } $}.
\end{align} 
Somewhat parsimoniously, the bound in \eqref{eq:ineq} both has order $1/\sqrt{N}$ and provides coverage for $t$ in a region of the order $\sqrt{N}$. Indeed, by plugging $L_N =c  \frac{ \sqrt{N}}{\mathbb{E}[|X|^3] } $ into \eqref{eq:berry smoothing 2} we obtain 
\begin{align*}
\mathrm{d}_{\mathrm{KS}}\left( \frac{X_1+\ldots+X_N}{\sqrt{N}},G \right)  &\leq \frac{ 4  \mathbb{E}[|X|^3] }{ c \sqrt{N}} + \frac{C}{\pi} \frac{  \mathbb{E}[|X|^3]}{ \sqrt{N}}   \int_{ - L_N}^{L_N} t^2 e^{ - t^2/4}  \mathrm{d} t \\
& \leq \frac{C \mathbb{E}[|X|^3] }{\sqrt{N}},
\end{align*}
where to obtain the final inequality above, we simply used the fact that $\int_{-\infty}^\infty t^2e^{-t^2/4}\mathrm{d} t \leq C$. In short, by using \eqref{eq:ineq} in \eqref{eq:berry smoothing 2}, one obtains the Berry-Esseen inequality \eqref{eq:be}. 

In our proof of the more general Theorem \ref{thm:main}, we too appeal to the Berry smoothing inequality, though our proof is otherwise self-contained. However, the argument we use to prove Theorem \ref{thm:main} has additional facets to it. In the setting where we have higher matching of moments of $X$ with the standard Gaussian distribution (i.e.\ \eqref{eq:mm} holds with $k \geq 3$), we are able to improve on the bound in \eqref{eq:ineq}. Indeed, in Theorem \ref{thm:creat} we show that if $\mathbb{E}[X^j] = \mathbb{E}[G^j]$ for all $j \leq k$, and $\mathbb{E}[|X|^{k+1}]$ is finite, then 
\begin{align} \label{eq:ineq2}
| \varphi(t/\sqrt{N})^N - e^{ -t^2/2} | \leq \frac{4}{(k+1)!} \mathbb{E}[|X|^{k+1}] \frac{ |t|^{k+1}}{N^{\frac{k-1}{2}  } } e^{ - t^2/4} ~~ \qquad \text{for all $|t| \leq c_{k,X} \sqrt{N}$},
\end{align}
where $c_{k,X}$ is a constant depending on $k$ and on $\mathbb{E}[|X|^{k+1}]$. 

The equation \eqref{eq:ineq} is then the special case $k = 2$ of \eqref{eq:ineq2}. 

While, for $k \geq 3$ the bound in \eqref{eq:ineq2} is asymptotically sharper than \eqref{eq:ineq} in its range (with the bound of the order $N^{- \frac{k-1}{2}}$ improving on $N^{-1/2}$), it shares in common with the $k=2$ case in \eqref{eq:ineq} that the bound is active only for $|t|$ of the order up to $O(\sqrt{N})$. As such, with \eqref{eq:ineq2} alone as a tool to plug into \eqref{eq:berry smoothing 2}, the best choice of $L$ we can take is still only $O(\sqrt{N})$, and thus the best bound we can hope for is of the order $1/\sqrt{N}$.

In fact, for every $k$, one may construct a lattice-valued random variable whose first $k$ moments agree with that of the standard Gaussian distribution. (The Bernoulli random variable provides an example for $k=3$; for $k \geq 4$, the reader may want to consult literature such as \cite{A1} or \cite{A2}.) Distribution functions for rescaled sums of i.i.d.\ copies of such a random variable have jumps of the order $1/\sqrt{N}$, and hence the Berry-Esseen inequality for such random variables is sharp with this order. In summary, even with an arbitrary degree of moment matching, without additional smoothness assumptions, we may not improve on the $O(1/\sqrt{N})$ speed of convergence of the standard Berry-Esseen inequality.

All is not lost however. Under Assumption \ref{assu:main}, we gain additional control over the integrand $| \varphi(t/\sqrt{N})^N - e^{ -t^2/2} |$ over a larger region in $t$ than that supplied by \eqref{eq:ineq2}. Outside of the $O(\sqrt{N})$-sized region where \eqref{eq:ineq2} provides coverage, we simply use the triangle inequality $| \varphi(t/\sqrt{N})^N - e^{ -t^2/2} | \leq | \varphi(t/\sqrt{N})^N| + |e^{ -t^2/2} |$, and control each term on the right-hand-side individually. On the one hand we have the simple bound,
\begin{align*}
\int_{|t| \geq c\sqrt{N}}  |e^{ -t^2/2} | \mathrm{d}t/t \leq Ce^{ - c' N}.
\end{align*}
More importantly, a simple Fourier-analytic argument tells us that if $\varphi(t) := \int_{-\infty}^\infty e^{its} \mu(\mathrm{d}s)$ is the characteristic function of a probability law $\mu$ satisfying Assumption $1$, then there is a universal constant $c$ such that 
\begin{align} \label{eq:away}
|\varphi(t)| \leq e^{ - c hw^3 t_0^2} \qquad \text{for all $|t| \geq t_0$}.
\end{align}
In particular, it follows that 
\begin{align*}
|\varphi(t/\sqrt{N})|^N \leq e^{ - chw^3 c_0^2 N} \qquad \text{for all $|t| \geq c_0 \sqrt{N}$}.
\end{align*}
After carefully overlaying the bounds, and selecting an optimal choice of $L$, we can control the integrand in \eqref{eq:berry smoothing 2} with a bound of the order $N^{-\frac{k-1}{2}}$, over a region of order at least $N^{\frac{k-1}{2}}$, thereby proving Theorem \ref{thm:main}.

\color{black}

\subsection{Related work}

The Berry-Esseen inequality was proved independently by Berry \cite{berry} and Esseen \cite{esseen} in the early 1940s. In fact the Berry-Esseen inequality \eqref{eq:be} is valid for independent but nonidentically distributed random variables with $\mathbb{E}[|X_j|^3] <\infty$. The key ideas of our proof of Theorem \ref{thm:main} all carry through to the non-identically distributed case (say, if we assume each random variable has a law satisfying Assumption \ref{assu:main}), but for the sake of fluency we have concentrated on the identically distributed case. 

As mentioned above, the exact statement we give in Theorem \ref{thm:main} is new, but there are similar results in the literature, usually presenting the bounds with terms involving the characteristic function of the constituent random variables. Most notably, we have Osipov's theorem \cite[Theorem 5.18]{petrov} which in the case that $X_1,\ldots,X_N$ are i.i.d.\ and satisfy \eqref{eq:mm} for some $k \geq 3$, states that 
\begin{align} \label{eq:Osipov}
&\sup_{s \in \mathbb{R}} \left| \mathbb{P} \left(  \frac{X_1 + \ldots + X_N}{ \sqrt{N} } \leq s \right) -  \int_{ - \infty}^s e^{ - u^2/2} \frac{ \mathrm{d} u }{ \sqrt{2 \pi}} \right| \nonumber \\
&\leq c(k) \left\{ \frac{\mathbb{E}[|X|^{k+1}]}{  N^{  \frac{k-1}{2}} } + N^{k(k+1)/2} \left( \sup_{|t|>\delta} |\varphi(t)| + 1/2N \right)^N \right\},
\end{align}
where $\delta := c/\mathbb{E}[|X|^3]$ for a universal $c>0$. (We have simplified the statement slightly for brevity - Osipov's theorem more generally involves an expansion in terms of the moments of $X_1$, which need not match the normal.) On this front we also mention work by Yaroslavtseva, e.g.\ \cite{yar}. 

Suppose Assumption \ref{assu:main} holds. Then after an appeal to Lemma \ref{lem:charcont} of the present article, and some further calculation, this reduces to 
\begin{align} \label{eq:Osipov2}
\sup_{s \in \mathbb{R}} \left| \mathbb{P} \left(  \frac{X_1 + \ldots + X_N}{ \sqrt{N} } \leq s \right) -  \int_{ - \infty}^s e^{ - u^2/2} \frac{ \mathrm{d} u }{ \sqrt{2 \pi}} \right| \leq c(k) \left\{ \frac{\mathbb{E}[|X|^{k+1}]}{  N^{  \frac{k-1}{2}} } + N^{k(k+1)/2} e^{ - chw^3N/\mathbb{E}[|X|^3]^2} \right\}.
\end{align}
Aside from the additional $N^{k(k+1)/2}$ factor (which, for large $N$, is asymptotically negligable against the exponential term) this recovers Theorem \ref{thm:main}. While Theorem \ref{thm:main} follows as a consequence of \eqref{eq:Osipov} and the fairly straightforward Lemma \ref{lem:charcont} of the present article, we nonetheless choose to prove Theorem \ref{thm:main} in full assuming only \eqref{eq:berry smoothing 1}, as the proof (inclusive of Lemma \ref{lem:charcont}) occupies only five pages. 

There are several other related results, often providing more detailed information than Theorem \ref{thm:main} but under stronger conditions. 
Like in the present article, Boutsikas \cite{boutsikas} studies random variables with densities and whose first $k$ moments match the standard Gaussian. 
Under a fairly detailed condition on the total variation distance between the law of $X_1$ and a standard Gaussian distribution, Boutsikas shows that the asymptotic total variation distance between the recentered sum $S_N = N^{-1/2}(X_1+\ldots+X_N)$ and a standard Gaussian random variable decays like $N^{-\frac{k-1}{2}}$. 
In a recent preprint, Derumigny, Girard and Guyonvarch \cite{DGG} prove a result similar to Theorem \ref{thm:main}, but under a complicated Fourier-analytic	 condition, which is satisfied in particular
when the density function of $S_N$ is $(N-1)$-times differentiable. The implied constants in the theorem then depend on this $(N-1)^{\text{th}}$ derivative. There is also a body of work involving optimising the bounds in the absense of a third moment condition but instead a condition of the form $\mathbb{E}[|X|^{2+\delta}] < \infty$ for some $\delta \in (0,1)$. See e.g.\ Bobkov \cite{bobkov} and the references therein. We also mention Jirak \cite{jirak}, who has studied the speed of convergence for stationary sequences.

Another relevant work in a similar spirit is that of Klartag and Sodin \cite{KS}, who draw on connections with high-dimensional geometry in their study of a weighted version of the Berry-Esseen inequality. Namely, let $X_1,\ldots,X_N$ be independent random variables distributed like $X$, where $X$ has finite fourth moment. Let $\theta = (\theta_1,\ldots,\theta_N)$ be an element of the unit sphere $\mathbb{S}^{N-1}$ (i.e.\ $\sum_{j=1}^N\theta_j^2 =1$). Klartag and Sodin use an intricate Fourier-analytic argument to show that, roughly speaking, for most elements $\theta$ of the unit sphere we have
\begin{align*}
\mathrm{d}_{\mathrm{KS}}( \theta_1 X_1 + \ldots + \theta_N X_N, G) \leq C\mathbb{E}[X^4]/N
\end{align*} 
exceeding the $1/\sqrt{N}$ rate given by the standard Berry-Esseen inequality \eqref{eq:be}, which corresponds to the case $\theta_1 = \ldots = \theta_N = 1/\sqrt{N}$. The intuition for this result is related to Maxwell's principle, which states that the coordinates of a uniformly chosen element of the unit sphere are approximately Gaussian \cite{DF2, JP}. Thus, for a typical element $\theta$ of the unit sphere, $\theta_1X_1 + \ldots + \theta_n X_N$ is `more Gaussian' than the equally weighted sum $\frac{1}{\sqrt{N}}X_1 + \ldots + \frac{1}{\sqrt{N}}X_N$.

Finally, let us touch on the local limit theorem, which is the analogue of the Berry-Esseen inequality for the density function (as opposed to distribution function) of a rescaled sum of independent and identically distributed random variables. The local limit theorem (see e.g.\ Section XVI.2 of Feller \cite{feller}) states that if $X$ has a density with respect to Lebesgue measure, and that some positive power $|\varphi(t)|^\nu$ of the modulus of the characteristic function of $X$ is integrable, then the density $f_n(s)$ of $N^{-1/2}(X_1+\ldots+X_N)$ (with $X_i \sim^{\text{i.i.d.}} X$) satisfies
\begin{align} \label{eq:edgeworth}
f_n(s) = \frac{1}{\sqrt{2\pi}}e^{-s^2/2} \left\{ 1 + \frac{\mathbb{E}[X^3]}{6\sqrt{N} } (s^3-3s) \right\} + o(N^{-1/2}),
\end{align}
where the error term $o(N^{-1/2})$ is uniform in $s$. Of course, in the setting $k \geq 3$ in \eqref{eq:mm}, we would have $\mathbb{E}[X^3]=0$, so that the density of $f_n(s)$ agrees with the standard Gaussian density up to an error $o(N^{-1/2})$. This result complements Theorem \ref{thm:main} in the case $k =3$. We should mention that it is possible to obtain further terms in the expansion \eqref{eq:edgeworth} in terms of Hermite polynomials; the full expansion is known as the Edgeworth expansion. 
 
In recent decades, Berry-Esseen theory has remained an extremely active field of research \cite{BCG, GS, XGL, DF, bobkov, raic}. We highlight in particular Shetsova \cite{shev1,shev2,shev3,shev4} and her coauthors, who have worked on optimising the constant $C$ occuring in \eqref{eq:be} in both the identically and non-identically distributed case.

\subsection{Overview}
 The remainder of the paper is structured as follows. In Section \ref{sec:local} we work on local control over characteristic functions, culminating in a proof of Theorem \ref{thm:creat}, which is a more explicit version of the inequality outlined in \eqref{eq:ineq2}. In Section \ref{sec:global} we study global control over characteristic functions, showing that probability laws satisfying Assumption \ref{assu:main} satisfy an equation of the form \eqref{eq:away}.\color{black} In Section \ref{sec:mainproof}, we tie our work together, leading to a proof of Theorem \ref{thm:main}. In the final part, Section \ref{sec:reverse}, we prove Theorem \ref{thm:reverse} showing that the $hw^3$ power in Theorem \ref{thm:main} is optimal.

In our derivation of Theorem \ref{thm:main}, we make a mild effort to control the constants.

At several points through the article we will make use of the well known Stirling inequalities
\begin{align} \label{eq:stirling}
\sqrt{2 \pi k}(k/e)^k \leq k! \leq \frac{11}{10} \sqrt{2 \pi k}(k/e)^k,
\end{align}
which are valid for all $k \geq 1$. 
\section{Local control of characteristic functions with moment matching} \label{sec:local}
In this section, we work towards and prove Theorem \ref{thm:creat}, which is a bound of the form \eqref{eq:ineq2}. Throughout this section we will be working under a moment matching bound, but emphasise that at this stage we do not make any assumptions on the smoothness of the density of the random variables. 

We use the notation $\mathrm{Log}:\mathbb{C} \to \mathbb{R} \times i(-\pi,\pi]$ for the principal value of the complex logarithm.

\color{black}
We begin with the following lemma, which amounts to the case $N=1$ of Theorem \ref{thm:creat}.
\begin{lemma} \label{lem:local1}
Let $\varphi(t) := \mathbb{E}[e^{itX}]$ be the characteristic function of a random variable $X$ satisfying \eqref{eq:mm} for some $k\geq 2$. Then
\begin{align*}
| \varphi(t) - e^{ -t^2/2} | \leq  3 \frac{ |t|^{k+1}}{(k+1)!} \mathbb{E}[|X|^{k+1} ] \qquad \text{for $t \in \mathbb{R}$}.
\end{align*}
\end{lemma}
\begin{proof}
For $u \in \mathbb{R}$, we have the Taylor series expansion $e^{iu} = \sum_{j=0}^k (iu)^j/j! + \Omega(u)|u|^{k+1}/(k+1)!$, where for each $u$, $\Omega(u)$ is a complex number of modulus at most $1$. In particular, with $G$ standard Gaussian, using \eqref{eq:mm} we have 
\begin{align*}
\varphi(t) - e^{ -t^2/2} = \mathbb{E} [ e^{itX} - e^{itG} ] = \frac{ |t|^{k+1}}{(k+1)!} \mathbb{E}\left[ \Omega(tX)|X|^{k+1} - \Omega(tG) |G|^{k+1} \right]\qquad \text{for } t \in \mathbb{R}.
\end{align*}
Thus by the triangle inequality
\begin{align} \label{eq:crel}
| \varphi(t) - e^{ -t^2/2} | \leq  \frac{ |t|^{k+1}}{(k+1)!} ( \mathbb{E}[|X|^{k+1} + \mathbb{E}[|G|^{k+1}]) \qquad \text{for } t \in \mathbb{R}.
\end{align}
A brief calculation using the Stirling bounds \eqref{eq:stirling} and $\mathbb{E}[|G|^j] = 2^{j/2}\Gamma((j+1)/2)$ tells us that $\mathbb{E}[|G|^{k+1}] \leq 2 \mathbb{E}[ |G|^k ]^{\frac{k+1}{k}}= 2 \mathbb{E}[ |X|^k ]^{\frac{k+1}{k}}$, which, by Jensen's inequality, is less than $2 \mathbb{E}[ |X|^{k+1}]$. Thus $\mathbb{E}[|G|^{k+1}] \leq 2 \mathbb{E}[|X|^{k+1}]$. Using this fact in \eqref{eq:crel} completes the proof.
\end{proof}

Let $m_{k+1}:= \mathbb{E}[|X|^{k+1}]^{\frac{1}{k+1}}$. In our proof of the following lemma we will make a reasonable effort to optimise constants, since we will be carrying these constants with us for the remainder of the article.
\begin{lemma} \label{lem:2}
Let $\varphi(t) := \mathbb{E}[e^{itX}]$ be the characteristic function of a random variable $X$ satisfying \eqref{eq:mm} for some $k\geq 2$. Then
\begin{align*}
|\mathrm{Log} \varphi(t) + t^2/2| \leq 4 \frac{ |t|^{k+1}}{(k+1)!} \mathbb{E}[|X|^{k+1} ] \qquad \text{for all $|t| \leq c_1 := \frac{1}{2} \wedge \frac{k+1}{4m_{k+1}}$}.
\end{align*}
\end{lemma}
\begin{proof}
We would like to write $\mathrm{Log} \varphi(t) + t^2/2 = \mathrm{Log} ( 1 + e^{t^2/2}(\varphi(t) - e^{-t^2/2}))$, and then expand $\mathrm{Log}(1 + z)$ as a power series in $z=e^{t^2/2}(\varphi(t) - e^{-t^2/2})$. To this end, we note that
\begin{align} \label{eq:logexp}
|\mathrm{Log} (1 + z)| \leq \frac{4}{3}|z| \qquad \text{whenever $|z|\leq 1/4$}.
\end{align}
Using the lower bound in \eqref{eq:stirling} and the fact that $k+1 \geq 3$, we have $(k+1)! \geq \sqrt{6 \pi } ((k+1)/e)^{(k+1)} $. Hence by Lemma \ref{lem:local1} we have 
\begin{align*}
| \varphi(t) - e^{ -t^2/2} | \leq  \frac{3}{\sqrt{6 \pi}} \left( \frac{m_{k+1}e|t|}{k+1} \right)^{k+1},
\end{align*}
for $t \in \mathbb{R}$.
Thus whenever $|t| \leq \frac{k+1}{4m_{k+1}}$, again using $k+1 \geq 3$ we have
\begin{align*}
| \varphi(t) - e^{ -t^2/2} | \leq  \frac{3}{\sqrt{6 \pi}} (e/4)^3.
\end{align*}
If $|t| \leq 1/2$, we also have $e^{t^2/2} \leq e^{1/8}$. In particular, we have
\begin{align} \label{eq:c1}
e^{t^2/2}| \varphi(t) - e^{ -t^2/2} | \leq  \frac{3}{\sqrt{6 \pi}} (e/4)^3 e^{1/8} \leq  1/4 \qquad \text{whenever $|t| \leq c_1$}.
\end{align}
Using \eqref{eq:c1}, \eqref{eq:logexp} and Lemma \ref{lem:local1}, it follow that whenever $|t| \leq c_1$ we have
\begin{align*}
\left| \mathrm{Log} \varphi(t) + t^2 \right| \leq 4 \frac{|t|^{k+1}}{(k+1)!} \mathbb{E}[|X|^{k+1}],
\end{align*}
thereby completing the proof.
\end{proof}

The following corollary follows quickly from the previous lemma.
\begin{cor} \label{cor:arrau}
Let $\varphi(t) := \mathbb{E}[e^{itX}]$ be the characteristic function of a random variable $X$ satisfying \eqref{eq:mm} for some $k\geq 2$. Then
\begin{align*}
|\mathrm{Log} \varphi(t/\sqrt{N})^N + t^2/2| \leq 4 N^{ - \frac{k-1}{2}} \frac{ |t|^{k+1}}{(k+1)!} \mathbb{E}[|X|^{k+1} ] \qquad \text{for all $|t| \leq c_1 \sqrt{N} $} .
\end{align*}
\end{cor}
\begin{proof}
Plainly 
\begin{align*}
|\mathrm{Log} \varphi(t/\sqrt{N})^N + t^2/2| = N| \mathrm{Log} \varphi(t/\sqrt{N}) + (t/\sqrt{N})^2/2 |.
\end{align*}
Now use the previous result, Lemma \ref{lem:2}, with $t/\sqrt{N}$ in place of $t$.
\end{proof}

We are now ready to state and prove the main result of this section, Theorem \ref{thm:creat}.

\begin{thm} \label{thm:creat}
Let $\varphi(t) := \mathbb{E}[e^{itX}]$ be the characteristic function of a random variable $X$ satisfying \eqref{eq:mm} for some $k\geq 2$. 
Let $c_0 = \frac{1}{2} \wedge \frac{1}{4} ((k+1)/m_{k+1}) \wedge 2 ((k+1)/m_{k+1})^{\frac{k+1}{k-1}}$. Then 
\begin{align} \label{eq:creat}
\left|  \varphi(t/\sqrt{N})^N - e^{-t^2/2} \right| \leq 4 N^{ - \frac{k-1}{2}}  \frac{ |t|^{k+1}}{(k+1)!} \mathbb{E}[|X|^{k+1} ]e^{ - t^2/4} \qquad \text{for all $|t| \leq c_0 \sqrt{N} $} .
\end{align}

\end{thm}

\begin{proof}
By the inequality $|e^z-1| \leq |z|e^{|z|}$ we have
\begin{align*}
\left|  \varphi(t/\sqrt{N})^N - e^{-t^2/2} \right| & =e^{-t^2/2} \left|  e^{t^2/2} \varphi(t/\sqrt{N})^N - 1\right| \\
& \leq |\mathrm{Log} \varphi(t/\sqrt{N})^N + t^2/2|  e^{ - t^2/2 + |\mathrm{Log} \varphi(t/\sqrt{N})^N + t^2/2| },
\end{align*}
for all $t \in \mathbb{R}$. 
For $|t| \leq  c_1\sqrt{N}$ by Corollary \ref{cor:arrau} we have 
\begin{align} \label{eq:xa}
&\left|  \varphi(t/\sqrt{N})^N - e^{-t^2/2} \right| \leq F(t) e^{ - t^2/2 + F(t)},
\end{align}
where $F(t):= 4 N^{ - \frac{k-1}{2}}  \frac{ |t|^{k+1}}{(k+1)!} \mathbb{E}[|X|^{k+1} ]$.

In addition to the requirement $|t| \leq c_1 \sqrt{N} $, we now search for a condition on $t$ which forces $F(t) \leq t^2/4$, so that the exponential term in \eqref{eq:xa} is bounded above by $e^{-t^2/4}$, thereby ensuring that \eqref{eq:xa} is actually an effective bound. \color{black} Extracting from \eqref{eq:stirling} the rough lower bound $(k+1)! \geq \sqrt{2 \pi} ((k+1)/e)^{k+1}$, we have 
\begin{align*}
\frac{F(t)}{t^2/4} \leq \frac{16}{\sqrt{2 \pi}} (|t|/\sqrt{N})^{k-1} (m_{k+1}e/(k+1))^{k+1}.
\end{align*}
The bound $\frac{F(t)}{t^2/4} \leq 1$ is then guaranteed to hold whenever we have
\begin{align*}
\frac{|t|}{\sqrt{N}} \leq \left( e^{k+1}\sqrt{2 \pi}/16 \right)^{\frac{1}{k-1}} \left( \frac{k+1}{m_{k+1}} \right)^{\frac{k+1}{k-1}}.
\end{align*}
The quantity $\left( e^{k+1}\sqrt{2 \pi}/16 \right)^{\frac{1}{k-1}} $ is decreasing in $k$, and is bounded below by $2$. Thus, setting $c_2 := 2 ((k+1)/m_{k+1} )^{\frac{k+1}{k-1}}$, the bound $|t| \leq c_2 \sqrt{N}$ guarantees $F(t) \leq t^2/4$. Accordingly, for $|t| \leq (c_1 \wedge c_2)\sqrt{N}$, by \eqref{eq:xa} we have 
\begin{align} \label{eq:xa2}
&\left|  \varphi(t/\sqrt{N})^N - e^{-t^2/2} \right| \leq F(t) e^{ - t^2/4}.
\end{align}
Since $c_0 = c_1 \wedge c_2$, \eqref{eq:creat} holds, completing the proof.
\end{proof}

\section{Global control of characteristic functions under Assumption \ref{assu:main}} \label{sec:global}

The following lemma says that away from the origin, the characteristic functions of random variables satisfying Assumption \ref{assu:main} have modulus bounded away from one.

\begin{lemma} \label{lem:charcont}
Let $\varphi(t) := \mathbb{E}[e^{itX}]$ be the characteristic function of a random variable $X$ whose law $\mu$ satisfies Assumption 1. Then provided $0 < t _0  \leq 4/w$ we have 
\begin{align*}
|\varphi(t)| \leq e^{ - hw^3t_0^2/32} \qquad \text{for all $|t| \geq t_0$}.
\end{align*}
\end{lemma}
\begin{proof}
Let $\xi$ be the measure on $\mathbb{R}$ given by $\xi(\mathrm{d}u) = h \mathrm{1}_{u \in [a,a+w]} \mathrm{d}u$. Then $\xi$ is a measure of total mass $hw$ and $\tilde{\mu} := \mu-\xi$ is a measure on $\mathbb{R}$ with total mass $1-hw$. 

Thus
\begin{align} \label{eq:schumann} 
|\varphi(t)| &\leq \left| \int_{-\infty}^\infty e^{ it u } \tilde{\mu}(\mathrm{d}u)  \right|  + \left| \int_{-\infty}^\infty e^{ it u } \xi(\mathrm{d}u) \right| \nonumber \\
&\leq 1 - hw + h \left| \int_a^{a+w} e^{ i t u } \mathrm{d} u \right| \nonumber \\
& = 1 - hw + h w \frac{ \sin(wt/2)}{wt/2} .
\end{align}
It is easily verified that we have both 
\begin{align*}
\frac{\sin(x)}{x} \leq 1 - x^2/8 \qquad |x| \leq 2, \qquad \text{and} \qquad \frac{\sin(x)}{x} \leq \frac{1}{|x|} \qquad |x| \geq 2.
\end{align*}
Together these bounds imply that for any $y \in [0,2]$ we have 
\begin{align} \label{eq:newky}
 \frac{\sin(x)}{x} \leq 1 - y^2/8 \qquad \text{for all $|x| \geq y$}.
\end{align}
Note that $t_0 \leq 4/w$ implies $y = wt_0/2 \in [0,2]$. In particular, using \eqref{eq:newky} in \eqref{eq:schumann}, for any $t$ with $|t| \geq t_0$ we have 
\begin{align*}
|\varphi(t)| \leq 1 - hw + hw (1 - (wt_0/2)^2/8) \leq 1 - hw^3 t_0^2/32 \leq e^{ - hw^3t_0^2/32},
\end{align*}
as required.
\end{proof}

\section{Proof of Theorem \ref{thm:main}} \label{sec:mainproof}

In this section we prove the following refinement of Theorem \ref{thm:main}. 

\begin{thm} \label{thm:main2}
Let $X_1,\ldots,X_N$ be independent random variables distributed like $X$, where $X$ satisfies \eqref{eq:mm} for some $k \geq 2$, and the law $\mu$ of $X$ satisfies Assumption \ref{assu:main}. Then
\begin{align} \label{eq:main2}
\mathrm{d}_{\mathrm{KS}}\left(\frac{X_1+\ldots+X_N}{\sqrt{N}},  G \right) \leq C(k) \frac{\mathbb{E}[|X|^{k+1}]}{N^{\frac{k-1}{2}}}  +  3e^{ - \tilde{c}(k,m_{k+1}) hw^3 N},
\end{align}
where 
\begin{align*}
C(k) :=  \frac{1}{\sqrt{\pi}} 2^{k+3} \Gamma((k+1)/2)/(k+1)!,
\end{align*}
and
\begin{align*}
\tilde{c}(k,m_{k+1}) := \frac{1}{160} \wedge \frac{1}{640} \left( \frac{k+1}{m_{k+1}} \right)^2 \wedge \frac{1}{10} \left( \frac{k+1}{m_{k+1}} \right)^{ 2 \frac{k+1}{k-1} }.
\end{align*}
\end{thm}
\begin{proof}
Recall from Theorem \ref{thm:creat} the constant $c_0 = \frac{1}{2} \wedge \frac{1}{4} ((k+1)/m_{k+1}) \wedge 2 ((k+1)/m_{k+1})^{\frac{k+1}{k-1}}$. Let us assume initially that $c_0 \sqrt{N} \geq 1$. 

By \eqref{eq:berry smoothing 2} and the triangle inequality, for any $L \geq 1$ we have 
\begin{align} \label{eq:berry smoothing 22}
\mathrm{d}_{\mathrm{KS}}\left(\frac{X_1+\ldots+X_N}{\sqrt{N}},  G \right) &\leq \frac{4}{L} + \frac{1}{\pi} (I_1 + I_2 + I_3),
\end{align}
where
\begin{align*}
 I_1 &:= \int_{|t| \leq c_0 \sqrt{N}} \frac{| \varphi(t/\sqrt{N})^N - e^{ -t^2/2} |}{|t|} \mathrm{d} t ,
\end{align*}
and
\begin{align*}
I_2 :=  \mathrm{1}_{\{ L \geq c_0\sqrt{N} \}} \int_{c_0 \sqrt{N} < |t| \leq L} e^{-t^2/2} \frac{\mathrm{d}t}{|t|}  \qquad \text{and} \qquad I_3 :=   \mathrm{1}_{\{ L \geq c_0\sqrt{N} \}} \int_{c_0 \sqrt{N} < |t| \leq L} |\varphi(t/\sqrt{N})|^N \frac{\mathrm{d}t}{|t|}.
\end{align*}
We will not specify the value of $L$ just yet, but instead control each of the terms individually for a free choice of $L$, and then later optimise this choice of $L$.

First we consider $I_1$. By Theorem \ref{thm:creat} we have
\begin{align*}
 I_1 &\leq 4N^{ - \frac{k-1}{2} } \mathbb{E}[|X|^{k+1}] \frac{1}{(k+1)!} \int_{ |t| \leq c_0 \sqrt{N}} |t|^{k+1} e^{-t^2/4} \frac{\mathrm{d}t}{|t|}.
\end{align*}
Now 
\begin{align*}
\int_{ |t| \leq c_0 \sqrt{N}} |t|^{k+1} e^{-t^2/4} \frac{\mathrm{d}t}{|t|} \leq 2^{\frac{k+1}{2}} \sqrt{2 \pi } \int_{-\infty}^\infty |u|^k \frac{e^{-u^2/2}\mathrm{d}u}{\sqrt{2\pi}} = \sqrt{\pi}2^{k+1} \Gamma((k+1)/2),
\end{align*}
where we have used the fact that if $G$ is standard Gaussian, then $\mathbb{E}[|G|^j] = 2^{j/2}\Gamma((j+1)/2)$. Thus,
\begin{align} \label{eq:cr1} 
\frac{1}{\pi} I_1 \leq C(k) \mathbb{E}[|X|^{k+1}] N^{ - \frac{k-1}{2}} 
\end{align}
where
\begin{align*}
C(k) = \frac{1}{\sqrt{\pi}} 2^{k+3} \Gamma((k+1)/2)/(k+1)!,
\end{align*}
as in the statement of the theorem.

As for $I_2$, whenever $c_0 \sqrt{N} \geq 1$ it is straightforward to verify that 
\begin{align} \label{eq:cr2}
I_2 \leq 2e^{ - \frac{c_0^2}{2}N}.
\end{align}
Finally, we consider $I_3$. Since $w \leq 1$ 
and $c_0 \leq 1/2$, we have $c_0 \leq 4/w$. Appealing to Lemma \ref{lem:charcont}, whenever $|t| \geq c_0 \sqrt{N}$ we then have
\begin{align*}
|\varphi(t/\sqrt{N})| \leq e^{ - hw^3c_0^2/32},
\end{align*}
so that 
\begin{align} \label{eq:cr3}
I_3 \leq 2 e^{ - hw^3c_0^2N/32} \mathrm{1}_{\{ L \geq c_0\sqrt{N} \}} \int_{c_0\sqrt{N}}^L \frac{\mathrm{d}t}{t}  \leq 2 e^{ - hw^3c_0^2N/32} \log(L),
\end{align}
where again we have used the fact that $c_0 \sqrt{N} \geq 1$. 

Aggregating \eqref{eq:cr1}, \eqref{eq:cr2} and \eqref{eq:cr3} in \eqref{eq:berry smoothing 22}, we see that provided $c_0 \sqrt{N} \geq 1$ we have
\begin{align} \label{eq:berry smoothing 33}
\mathrm{d}_{\mathrm{KS}}\left(\frac{X_1+\ldots+X_N}{\sqrt{N}},  G\right) \leq \frac{4}{L} + C(k) \frac{\mathbb{E}[|X|^{k+1}]}{N^{\frac{k-1}{2}}}  + \frac{2}{\pi} e^{ - \frac{c_0^2}{2}N} +  \frac{2}{\pi} e^{ - hw^3c_0^2N/32} \log L.
\end{align}
Letting $L = 2  \pi e^{ hw^3c_0^2 N /32 }$ in \eqref{eq:cr3} we obtain
\begin{align} \label{eq:berry smoothing 33}
\mathrm{d}_{\mathrm{KS}}\left(\frac{X_1+\ldots+X_N}{\sqrt{N}},  G\right) \leq C(k) \frac{\mathbb{E}[|X|^{k+1}]}{N^{\frac{k-1}{2}}}  + \frac{2}{\pi}(2 + \log(2\pi) +  \frac{1}{32} hw^3 c_0^2 N ) e^{ - hw^3c_0^2N/32},
\end{align}
where we have used the fact that $hw^3 \leq 1$ to smuggle the $e^{ - \frac{c_0^2}{2}N} $ term in with the $e^{ - hw^3c_0^2N/32}$ terms. 

Since $\frac{2}{\pi}(2 + \log(2\pi ) + x ) )e^{-x} \leq 3e^{-4x/5}$ for all $x \geq 0$, setting $x = hw^3c_0^2N/32$ we have the simplification
\begin{align} \label{eq:berry smoothing 44}
\mathrm{d}_{\mathrm{KS}}\left(\frac{X_1+\ldots+X_N}{\sqrt{N}},  G\right)
\leq C(k) \frac{\mathbb{E}[|X|^{k+1}]}{N^{\frac{k-1}{2}}}  + 3 e^{ - hw^3c_0^2N/40}.
\end{align}
Finally note that 
\begin{align*}
\frac{1}{40} c_0^2 = \frac{1}{40} \left\{ \frac{1}{4} \wedge \left( \frac{k+1}{4m_{k+1}} \right)^2 \wedge 4 \left( \frac{k+1}{m_{k+1}} \right)^{ 2 \frac{k+1}{k-1} } \right\} = \tilde{c}(k,m_{k+1} ),
\end{align*}
completing the proof of Theorem \ref{thm:main2} under the assumption that $c_0\sqrt{N} \geq 1$. Now note that if $c_0\sqrt{N} < 1$, since by definition $h,w \leq 1$, we have
\begin{align*}
3e^{ - \frac{hw^3}{40} c_0^2 N } \geq 3e^{ - \frac{1}{40} } \geq 1.
\end{align*} 
However, the Kolmogorov-Smirnov distance $\mathrm{d}_{\mathrm{KS}}\left(\frac{X_1+\ldots+X_N}{\sqrt{N}},  G\right) $ is at most $1$, hence the bound \eqref{eq:main2} continues to hold when $c_0\sqrt{N}<1$.
\end{proof}

\begin{proof}[Proof of Theorem \ref{thm:main} assuming Theorem \ref{thm:main2}]
To obtain the simpler statement from the introduction, Theorem \ref{thm:main}, first we note that for $x \geq 1$, $C(x)$ is a decreasing function of $x$, and that $C(3) \leq 3$.

Then we note that, since $m_{k+1} \geq 1$, and $(k+1) \geq 4$ we have
\begin{align*}
\tilde{c}(k,m_{k+1}) & \geq \frac{1}{160}m_{k+1}^{-(k+1)},
\end{align*}
completing the proof of Theorem \ref{thm:main}.
\end{proof}

Finally, with Theorem \ref{thm:main}, and its corollary, Corollary \ref{cor:symmetric} now established, we give a proof of \eqref{eq:main4}. 

\begin{proof}[Proof of \eqref{eq:main4}]
Let $N \geq 100000$. Then
\begin{align} \label{eq:giant}
\delta_N^2 = 16(\log N/N)^{1/2} \leq 0.2.
\end{align}
Recall the definition \eqref{eq:special} of $\mu_{h,w}$, and $x_{h,w}$ as it appears below \eqref{eq:special}. Plainly, for $w \leq 1$,
\begin{align} \label{eq:4up}
\int_{-\infty}^\infty s^4 \mu_{h,w}(\mathrm{d}s) \leq x_{h,w}^4 \leq (1 - hw)^{-2}.
\end{align}
In particular, setting $h = w= \delta_N$ and using \eqref{eq:giant}, we see that the fourth moment of $\nu_N$ is at most $(5/4)^2$.

It follows from Corollary \ref{cor:symmetric} that for $X_1,\ldots,X_N$ independent and identically distributed like $\nu_N$ we have
\begin{align} \label{eq:mous}
\mathrm{d}_{\mathrm{KS}}\left(\frac{X_1+\ldots+X_N}{ \sqrt{N}},  G\right) &\leq 3 \left\{ \frac{ (5/4)^2  }{N} + e^{ - \frac{(4/5)^2}{160 } 4^4 \log N } \right\} \leq 8/N,
\end{align}
as required.

\end{proof}

\section{Proof of Theorem \ref{thm:reverse}} \label{sec:reverse}

Recall the definition \eqref{eq:special} of the probability law $\mu_{h,w}$.
The main task in proving Theorem \ref{thm:reverse} is establishing the inequality \eqref{eq:counter}, which we state here as a lemma involving explicit constants.
\begin{lemma} \label{lem:counter}
Let $0 < h,w \leq 1$ such that $hw \leq 1/2$. Suppose $hw^3N \leq 1/24$. Then whenever $X_1,\ldots,X_N$ are i.i.d.\ with law $\mu_{h,w}$ we have
\begin{align*}
\mathrm{d}_{\mathrm{KS}}\left( \frac{X_1 + \ldots + X_N}{ \sqrt{N} }, G \right) \geq \frac{1}{50}\frac{1}{\sqrt{N}}.
\end{align*}

\end{lemma}
\begin{proof}

 The law $\mu_{h,w}$ is a mixture of two different basic probability laws. Indeed, setting $\varepsilon = hw$, we note that 
\begin{align*}
\mu_{h,w}(\mathrm{d}s) =  (1 - \varepsilon) \nu_x(\mathrm{d}s) + \varepsilon \xi_w(\mathrm{d}s),
\end{align*}
where $\nu_x$ and $\xi_w$ are probability laws, $\nu_x$ being the law of a Bernoulli random variable taking values $x$ and $-x$ each with probability $1/2$, and $\xi_w$ being the law of a random variable uniformly distributed on $[-w/2,w/2]$. 

Let $X_1,\ldots,X_N$ be independent random variables identically distributed according to $\mu_{h,w}$. Define
\begin{align*}
F_N(s) := \mathbb{P} \left( \frac{X_1 + \ldots + X_N}{\sqrt{N}} \leq s \right).
\end{align*}
We will study the rate at which $F_N(s)$ increases over small intervals of width of the order $1/\sqrt{N}$, showing that it increases over these intervals significantly faster than the standard Gaussian distribution function $\Phi(s)$. 

To this end, since each $X_i$ has the law of a mixture of two different probability laws, i.e.\ a Bernoulli and a uniform, it follows that the sum $N^{-1/2}(X_1+\ldots+X_N)$ can be realised as a binomial mixture of Bernoullis and uniforms. Namely,
\begin{align*}
F_N(s) = \sum_{j=0}^N \binom{N}{j}(1-\varepsilon)^{N-j}\varepsilon^j G_{N-j,j}(s),
\end{align*}
where
\begin{align*}
G_{N-j,j}(s) := \mathbb{P}\left( \frac{ Z_1 + \ldots + Z_{N-j} + Y_1 + \ldots + Y_j}{ \sqrt{N} } \leq s \right), 
\end{align*}
given $(Z_i)$ and $(Y_i)$ independent, and with $Z_i$ distributed like $\nu_x$ and $Y_i$ distributed like $\xi_w$. 

Consider that since the Gaussian density is $1/\sqrt{2\pi}$ at zero, up to lower order terms the associated Gaussian distribution function $\Phi(s)$ increases by $1/2\sqrt{2\pi N}$ over the interval $[-\frac{1}{4\sqrt{N}},\frac{1}{4\sqrt{N}}]$. We now show that provided $hw^3N$ is not too large, $F_N(s)$ has as a larger increase (albeit still over the order $1/\sqrt{N}$) over the same interval. In fact, for reasons that will become clear, we have to adjust the interval slightly according to whether $N$ is odd or even, so that instead we consider $\left[ \frac{\mathrm{1}_{N \text{ odd}} -1/4}{\sqrt{N}}, \frac{\mathrm{1}_{N \text{ odd}} +1/4}{\sqrt{N}}\right]$, where $\mathrm{1}_{N \text{ odd}}$ denotes the indicator function that $N$ is odd.

To study the increase over this range, note we can write
\begin{align} \label{eq:upi}
&F_N\left(\frac{\mathrm{1}_{N \text{ odd}} +1/4}{\sqrt{N}}\right) - F_N\left(\frac{\mathrm{1}_{N \text{ odd}} -1/4}{\sqrt{N}}\right) \nonumber \\
&=  \sum_{j=0}^N \binom{N}{j}(1-\varepsilon)^{N-j}\varepsilon^j \left( G_{N-j,j}\left(\frac{\mathrm{1}_{N \text{ odd}} +1/4}{\sqrt{N}}\right) - G_{N-j,j}\left(\frac{\mathrm{1}_{N \text{ odd}} -1/4}{\sqrt{N}}\right) \right).
\end{align}
Now note that 
\begin{align} \label{eq:even0}
 &G_{N-j,j}\left(\frac{\mathrm{1}_{N \text{ odd}} +1/4}{\sqrt{N}}\right) - G_{N-j,j}\left(\frac{\mathrm{1}_{N \text{ odd}} -1/4}{\sqrt{N}}\right)\nonumber\\ &= \mathbb{P} \left( |Z_1 + \ldots + Z_{N-j} + Y_1 + \ldots + Y_ j - \mathrm{1}_{N \text{ odd}} | < 1/4 \right ) \nonumber \\
& \geq \mathbb{P}( Z_1 + \ldots + Z_{N-j} = \mathrm{1}_{N \text{ odd}} ) \mathbb{P} (|Y_1 + \ldots +Y_j| < 1/4 ) ,
\end{align}
though of course the event $\{Z_1 + \ldots + Z_{N-j} = \mathrm{1}_{N \text{ odd}} \}$ only has positive probability when $N-j$ has the same parity as $N$, i.e. when $j$ is even. In this case, we now find lower bounds for each of the latter probabilities.

First we consider the term involving $Z_i$. If $k$ is even, then recalling the Stirling inequalities \eqref{eq:stirling} we have 
\begin{align*}
\mathbb{P}( Z_1 + \ldots + Z_k = 0 ) = \frac{k!}{(k/2)!^2} 2^{-k} \geq \left( \frac{10}{11} \right)^2 \frac{ \sqrt{2 \pi k } (k/e)^k}{ 2 \pi k/2 (k/2e)^k} 2^{ - k} \geq \frac{3}{2} \frac{1}{\sqrt{2 \pi k}}.
\end{align*} 
Similarly, if $k$ is odd, we have
\begin{align*}
\mathbb{P}( Z_1 + \ldots + Z_k = 1 ) = \frac{k!}{\frac{k+1}{2}!\frac{k-1}{2}!} 2^{-k} \geq \left( \frac{10}{11} \right)^2 \frac{ \sqrt{2 \pi k } (k/e)^k}{ 2 \pi \sqrt{k+1}\sqrt{k-1} (k/2e)^k} 2^{ - k} \geq \frac{3}{2} \frac{1}{\sqrt{2 \pi k}}.
\end{align*} 
Thus, whenever $j$ is even we have 
\begin{align} \label{eq:even1}
\mathbb{P}( Z_1 + \ldots + Z_{N-j} = \mathrm{1}_{N \text{ odd}} ) \geq \frac{3}{2} \frac{1}{\sqrt{2 \pi N}}.
\end{align}
On the other hand, a calculation tells us that $Y_1$ has variance $w^2/12$. Taking a rough bound using Chebyshev's inequality we have 
\begin{align} \label{eq:even2}	
\mathbb{P} ( |Y_1 + \ldots + Y_ j | < 1/4 ) \geq 1 - 2 j w^2.
\end{align}
Combining \eqref{eq:even1} and \eqref{eq:even2} in \eqref{eq:even0}, we see that whenever $j$ is even we have
\begin{align} \label{eq:even4}
 &G_{N-j,j}\left(\frac{\mathrm{1}_{N \text{ odd}} +1/4}{\sqrt{N}}\right) - G_{N-j,j}\left(\frac{\mathrm{1}_{N \text{ odd}} -1/4}{\sqrt{N}}\right) \geq (1-2jw^2)\frac{3}{2} \frac{1}{\sqrt{2 \pi N}}.
\end{align}
Using \eqref{eq:even4} in \eqref{eq:upi}, we have 
\begin{align} \label{eq:upi2}
F_N\left(\frac{\mathrm{1}_{N \text{ odd}} +1/4}{\sqrt{N}}\right) - F_N\left(\frac{\mathrm{1}_{N \text{ odd}} -1/4}{\sqrt{N}}\right) \geq \frac{3}{2} \frac{1}{\sqrt{2\pi N}} \sum_{j=0, j \text{ even}}^N \binom{N}{j}(1-\varepsilon)^{N-j }\varepsilon^j  ( 1 - 2w^2j) .
\end{align}
Letting $J$ be a Binomial random variable with parameters $N$ and $\varepsilon$, using the simple bound $\mathbb{E}[J\mathrm{1}_{J \text{ even}}] \leq \mathbb{E}[J]$, from \eqref{eq:upi2} we have 
\begin{align} \label{eq:upi3}
F_N\left(\frac{\mathrm{1}_{N \text{ odd}} +1/4}{\sqrt{N}}\right) - F_N\left(\frac{\mathrm{1}_{N \text{ odd}} -1/4}{\sqrt{N}}\right) \geq \frac{3}{2} \frac{1}{\sqrt{2\pi N}} ( \mathbb{P}( J \text{ is even} ) - 2w^2 \mathbb{E}[J] ).
\end{align}
The probability that $J$ is even is $\frac{1}{2} + \frac{1}{2}(1 - 2 \varepsilon)^N$, which is greater than $1/2$ by our assumption $\varepsilon = hw \leq 1/2$. Also, $\mathbb{E}[J] = N\varepsilon$. Thus from \eqref{eq:upi3} we have 
\begin{align} \label{eq:upi4}
F_N\left(\frac{\mathrm{1}_{N \text{ odd}} +1/4}{\sqrt{N}}\right) - F_N\left(\frac{\mathrm{1}_{N \text{ odd}} -1/4}{\sqrt{N}}\right) \geq \frac{3}{2} \frac{1}{\sqrt{2\pi N}} ( \frac{1}{2} - 2w^2 N \varepsilon).
\end{align}
Conversely, note that
\begin{align} \label{eq:upi5}
\Phi\left(\frac{\mathrm{1}_{N \text{ odd}} +1/4}{\sqrt{N}}\right) - \Phi\left(\frac{\mathrm{1}_{N \text{ odd}} -1/4}{\sqrt{N}}\right) \leq \frac{1}{2 \sqrt{2 \pi N }}.
\end{align}

If $F,G:\mathbb{R} \to [0,1]$ are functions with $F(b) - F(a) = x$ and $G(b) - G(a) = y$, then
\begin{align*}
(G(b)-F(b))-(G(a)-F(a)) = y-x.
\end{align*}
In particular, either $|G(b)-F(b)|$ or $|G(a)-F(a)|$ exceeds $\frac{1}{2}|y-x|$. 

Whenever $1-12w^2N\varepsilon \geq 0$, the lower bound in \eqref{eq:upi4} exceeds the upper bound in \eqref{eq:upi5}. In particular, under this condition, by  \eqref{eq:upi4} and \eqref{eq:upi5} there exists $s \in \mathbb{R}$ such that 
\begin{align*}
|F_N(s) - \Phi(s) | \geq \frac{1}{2} \left( \frac{3}{2} \frac{1}{\sqrt{2\pi N}} ( \frac{1}{2} - 2w^2 N \varepsilon) - \frac{1}{2 \sqrt{2 \pi N }} \right) = \frac{1}{8\sqrt{2 \pi N}} (1 - 12hw^3N),
\end{align*}
where in the final equality above, we have used the fact that $\varepsilon = hw$. In particular, whenever $hw^3N \leq 1/24$, we have
\begin{align*}
\sup_{s \in \mathbb{R}}|F_N(s) - \Phi(s) | \geq \frac{1}{16\sqrt{2 \pi N}}  \geq \frac{1}{50\sqrt{N}}, 
\end{align*}
completing the proof.
\end{proof}

We are now ready to prove Theorem \ref{thm:reverse}.
\begin{proof}[Proof of Theorem \ref{thm:reverse}]
Let $C,c,\rho,\rho'$ be as in the statement of the theorem. Set 
\begin{align} \label{eq:scaa}
h = w = \frac{1}{3}N^{-1/4}.
\end{align}
Note $hw^3N = \frac{1}{81} \leq \frac{1}{24}$, so that whenever $X_1,\ldots,X_N$ are i.i.d.\ with law $\mu_{h,w}$, by Lemma \ref{lem:counter} we have 
\begin{align} \label{eq:olv1}
\mathrm{d}_{\mathrm{KS}}\left( \frac{X_1 + \ldots + X_N}{ \sqrt{N} }, G \right) \geq \frac{1}{50}\frac{1}{\sqrt{N}}.
\end{align}
We now study the quantity of the right-hand-side of \eqref{eq:reverse} as it pertains to $\mu_{h,w}$ with $h,w$ as in \eqref{eq:scaa}, setting $k=3$ as our degree of moment matching. By \eqref{eq:4up} we have
\begin{align*}
\mathbb{E}[|X|^4] \leq (1 - \frac{1}{9}N^{-1/2})^2 \leq 2,
\end{align*}
for all $N \geq 1$. Thus with $C,c',\rho,\rho'$ prescribed, and $k=3$, using $\mathbb{E}[|X|^4] \leq 2$ to obtain the first inequality below, and \eqref{eq:scaa} to obtain the second, the right-hand-side of \eqref{eq:reverse} satisfies
\begin{align} \label{eq:olv2}
 C \left\{ \frac{\mathbb{E}[|X|^{k+1}]}{N^{\frac{k-1}{2}}}  +  e^{ - c h^{1-\rho}w^{3-\rho'} N/\mathbb{E}[|X|^{k+1}]} \right\}&\leq C \left\{ \frac{2}{N} + e^{ - \frac{c}{2} h^{1 - \rho} w^{3 - \rho'} N } \right\}\nonumber \\
&\leq C \left\{ \frac{2}{N} + e^{ - \frac{c}{2}3^{\rho+\rho'-4}N^{(\rho+\rho')/4} } \right\} .
\end{align}
Now, if at least one of $\rho$ or $\rho'$ is positive, for all sufficiently large $N$ we have
\begin{align} \label{eq:olv3}
\frac{1}{50}\frac{1}{\sqrt{N}} > C \left\{ \frac{2}{N} + e^{ - \frac{c}{2} 3^{\rho+\rho'-4}N^{(\rho+\rho')/4} } \right\} .
\end{align}
Combining \eqref{eq:olv3} with \eqref{eq:olv1} and \eqref{eq:olv2}, for all sufficiently large $N$ we have 
\begin{align*}
\mathrm{d}_{\mathrm{KS}}\left( \frac{X_1 + \ldots + X_N}{ \sqrt{N} }, G \right) > C \left\{ \frac{\mathbb{E}[|X|^{k+1}]}{N^{\frac{k-1}{2}}}  +  e^{ - c h^{1-\rho}w^{3-\rho'} N/\mathbb{E}[|X|^{k+1}]} \right\},
\end{align*}
whenever $X_1,\ldots,X_N$ are i.i.d.\ with law $\mu_{h,w}$ (with $h,w$ as in \eqref{eq:scaa}, $k=3$, and $\mathbb{E}[|X|^4]$ being the fourth moment associated with $\mu_{h,w}$). That completes the proof of Theorem \ref{thm:reverse}.
\end{proof}

\end{document}